\documentclass[12pt]{amsart}

\usepackage{amssymb}
\usepackage{amscd}

\title[A birational embedding with two Galois points]{A birational embedding with two Galois points for quotient curves} 
\author{Satoru Fukasawa \and Kazuki Higashine}

\address{Department of Mathematical Sciences, Faculty of Science, Yamagata University, Kojirakawa-machi 1-4-12, Yamagata 990-8560, Japan} 
\email{s.fukasawa@sci.kj.yamagata-u.ac.jp} 
\address{Graduate School of Science and Engineering, Yamagata University, Kojirakawa-machi 1-4-12, Yamagata 990-8560, Japan}
\email{s16m111m@st.yamagata-u.ac.jp}

\thanks{The first author was partially supported by JSPS KAKENHI Grant Number 16K05088.}

\newtheorem{theorem}{Theorem}

\newtheorem{corollary}{Corollary} 
\newtheorem{fact}{Fact}

\theoremstyle{definition}
\newtheorem{remark}{Remark}

\begin{document}
\begin{abstract} 
A criterion for the existence of a birational embedding with two Galois points for quotient curves is presented. 
We apply our criterion to several curves, for example, some cyclic subcovers of the Giulietti--Korchm\'{a}ros curve or of the curves constructed by Skabelund. 
They are new examples of plane curves with two Galois points.  
\end{abstract}

\maketitle 

\section{Introduction} 

The notion of {\it Galois point} was introduced by Hisao Yoshihara  in 1996 (\cite{miura-yoshihara, yoshihara}): for a plane curve, a point in $\mathbb{P}^2$ is called a Galois point, if the function field extension induced by the projection is a Galois extension. 
If a Galois point is a smooth point, then we call it an inner Galois point. 

Let $\mathcal{C}$ be a (reduced, irreducible) smooth projective curve over an algebraically closed field $k$ of characteristic $p \ge 0$ and let $k(\mathcal{C})$ be its function field. 
Recently, a criterion for the existence of a birational embedding with two Galois points was presented by the first author (\cite{fukasawa1}), and by this criterion, several new examples of plane curves with two Galois points were described. 
We recall this criterion. 

\begin{fact} \label{criterion} 
Let $G_1$ and $G_2$ be finite subgroups of ${\rm Aut}(\mathcal{C})$ and let $P_1$ and $P_2$ be different points of $\mathcal{C}$.
Then, three conditions
\begin{itemize}
\item[(a)] $\mathcal{C}/{G_1} \cong \Bbb P^1$ and $\mathcal{C}/{G_2} \cong \Bbb P^1$,    
\item[(b)] $G_1 \cap G_2=\{1\}$, and
\item[(c)] $P_1+\sum_{\sigma \in G_1} \sigma (P_2)=P_2+\sum_{\tau \in G_2} \tau (P_1)  $
\end{itemize}
are satisfied, if and only if there exists a birational embedding $\varphi: \mathcal{C} \rightarrow \mathbb P^2$ of degree $|G_1|+1$ such that $\varphi(P_1)$ and $\varphi(P_2)$ are different inner Galois points for $\varphi(\mathcal{C})$ and the associated Galois group $G_{\varphi(P_i)}$ is equal to $G_i$ for $i=1, 2$. 
\end{fact}

Some known examples of plane curves with two Galois points are obtained as quotient curves $\mathcal{C}/H$ of curves $\mathcal{C}$ with a subgroup $H \subset {\rm Aut}(\mathcal{C})$ such that $\mathcal{C}$ has  a birational embedding with two Galois points. 
Typical examples are quotient curves of the Hermitian curve  (\cite{fukasawa-higashine1, homma}), and the Hermitian curve as a Galois subcover of the Giulietti--Korchm\'{a}ros curve (\cite{fukasawa2, fukasawa-higashine2}). 

Motivated by this observation, the aim of this article is to present a criterion for the existence of a plane model with two Galois points for quotient curves. 
For a finite subgroup $H$ of ${\rm Aut}(\mathcal{C})$ and a point $Q \in \mathcal{C}$, the quotient map is denoted by $f_H: \mathcal{C} \rightarrow \mathcal{C}/H$ and the image $f_H(Q)$ is denoted by $\overline{Q}$. 
Furthermore, if $H$ is a normal subgroup of a subgroup $G \subset {\rm Aut}(\mathcal{C})$, then there exists a natural homomorphism $G \rightarrow {\rm Aut}(\mathcal{C}/H)$.  
The image is denoted by $\overline{G}$, which is isomorphic to $G/H$. 

\begin{theorem} \label{main} 
Let $H$, $G_1$, and $G_2 \subset {\rm Aut}(\mathcal{C})$ be finite subgroups with $H \vartriangleleft G_i$ for $i=1, 2$, and let $P_1$ and $P_2 \in \mathcal{C}$. 
Then, four conditions
\begin{itemize}
\item[(a')] $\mathcal{C}/{G_1} \cong \Bbb P^1$ and $\mathcal{C}/{G_2} \cong \Bbb P^1$,    
\item[(b')] $G_1 \cap G_2=H$, 
\item[(c')] $\sum_{h \in H}h(P_1)+\sum_{\sigma \in G_1} \sigma (P_2)=\sum_{h \in H}h(P_2)+\sum_{\tau \in G_2} \tau (P_1)$, and
\item[(d')] $HP_1 \ne HP_2$
\end{itemize}
are satisfied, if and only if there exists a birational embedding $\varphi: \mathcal{C}/H \rightarrow \mathbb P^2$ of degree $|G_1/H|+1$ such that $\varphi(\overline{P_1})$ and $\varphi(\overline{P_2})$ are different inner Galois points for $\varphi(\mathcal{C}/H)$ and $G_{\varphi(\overline{P_i})}=\overline{G_i}$ for $i=1, 2$. 
\end{theorem}

\begin{remark}
A similar result holds for ``outer'' Galois points. 
In this case, we consider a $4$-tuple $(G_1, G_2, H, Q)$ with $Q \in \mathcal{C}$ such that $G_1$ and $G_2$ satisfy conditions (a') and (b'), and $\sum_{\sigma \in G_1}\sigma(Q)=\sum_{\tau \in G_2}\tau (Q)$ holds (see also \cite[Remark 1]{fukasawa1}). 
\end{remark}

As an application, for the case where $\mathcal{C}$ has a birational embedding with two Galois points, the following holds. 

\begin{corollary} \label{maincor} 
Let a $4$-tuple $(G_1, G_2, P_1, P_2)$ satisfies conditions {\rm (a)}, {\rm (b)} and {\rm (c)} in Fact \ref{criterion}, and let $H$ be a finite subgroup of ${\rm Aut}(\mathcal{C})$. 
If three conditions 
\begin{itemize}
\item[(d)] $H \cap G_1G_2=\{1\}$,  
\item[(e)] $HG_1=H \rtimes G_1$ and $HG_2=H \rtimes G_2$, and 
\item[(f)] $HP_1 \ne HP_2$
\end{itemize}
are satisfied, then there exists a birational embedding $\psi: \mathcal{C}/H \rightarrow \mathbb P^2$ of degree $|G_1|+1$ such that $\psi(\overline{P_1})$ and $\psi(\overline{P_2})$ are different inner Galois points for $\psi(\mathcal{C}/H)$ and $G_{\psi(\overline{P_i})} \cong G_i$ for $i=1, 2$. 
\end{corollary}

In Sections 3 and 4, we will apply Corollary \ref{maincor} to the Giulietti--Korchm\'{a}ros curve, and the curves constructed by Skabelund. 
Theorems \ref{gk}, \ref{skabelund1}, \ref{skabelund2} and \ref{skabelund3} provide new examples of plane curves with two Galois points (see the Table in \cite{open}). 
In Section 5, we discuss the relations between Corollary \ref{maincor} and the previous works. 

\section{Proof of the main theorem} 

\begin{proof}[Proof of Theorem \ref{main}]
We consider the only-if part. 
By condition (d'), $\overline{P_1} \ne \overline{P_2}$. 
We would like to prove that conditions (a), (b) and (c) in Fact \ref{criterion} are satisfied for the $4$-tuple $(\overline{G_1}, \overline{G_2}, \overline{P_1}, \overline{P_2})$. 
Since $k(\mathcal{C}/H)^{\overline{G_i}}=k(\mathcal{C})^{G_i}$, by condition (a'), the fixed field $k(\mathcal{C}/H)^{\overline{G_i}}$ is rational. 
It follows from condition (b') that $\overline{G_1}\cap\overline{G_2}=\{1\}$. 
Therefore, conditions (a) and (b) for the $4$-tuple $(\overline{G_1}, \overline{G_2}, \overline{P_1}, \overline{P_2})$ are satisfied. 
Since
$$ \sum_{\sigma \in G_1}\sigma(P_2)=\sum_{H\sigma \in G_1/H}\sum_{h \in H}h\sigma(P_2), $$
it follows that
$$ (f_H)_*\left(\sum_{\sigma \in G_1}\sigma(P_2)\right)=\sum_{H\sigma \in G_1/H}|H|\cdot\overline{\sigma(P_2)}=|H|\sum_{\overline{\sigma} \in \overline{G_1}}\overline{\sigma}(\overline{P_2}). $$
On the other hand, 
$(f_H)_*(\sum_{h \in H}h(P_1))=|H|\overline{P_1}$. 
It follows from condition (c') that 
$$|H|\left(\overline{P_1}+\sum_{\overline{\sigma} \in \overline{G_1}}\overline{\sigma}(\overline{P_2})\right)=|H|\left(\overline{P_2}+\sum_{\overline{\tau} \in \overline{G_2}}\overline{\tau}(\overline{P_1})\right). $$
Since $|H|\cdot D=0$ implies $D=0$ for any divisor $D$, we are able to cut the multiplier $|H|$. 
Condition (c) for the $4$-tuple $(\overline{G_1}, \overline{G_2}, \overline{P_1}, \overline{P_2})$ is satisfied. 

We consider the if part. 
We assume that conditions (a), (b) and (c) for the $4$-tuple $(\overline{G_1}, \overline{G_2}, \overline{P_1}, \overline{P_2})$ are satisfied.   
Since $k(\mathcal{C})^{G_i}=k(\mathcal{C}/H)^{\overline{G_i}}$, by condition (a), the fixed field $k(\mathcal{C})^{G_i}$ is rational. 
Condition (a') is satisfied. 
Since $\overline{G_1} \cap \overline{G_2}=\{1\}$, condition (b') is satisfied. 
Since $\varphi(\overline{P_1}) \ne \varphi(\overline{P_2})$, condition (d') is satisfied. 
By condition (c), 
$$\overline{P_1}+\sum_{\overline{\sigma} \in \overline{G_1}}\overline{\sigma}(\overline{P_2})=\overline{P_2}+\sum_{\overline{\tau} \in \overline{G_2}}\overline{\tau}(\overline{P_1}). $$
Then, 
\begin{eqnarray*} 
f_H^*\left(\overline{P_1}+\sum_{\overline{\sigma} \in \overline{G_1}}\overline{\sigma}(\overline{P_2})\right)&=&f_H^*(\overline{P_1})+\sum_{\overline{\sigma} \in \overline{G_1}}f_H^*(\overline{\sigma}(\overline{P_2})) \\
&=&\sum_{h \in H}h(P_1)+\sum_{H\sigma \in G_1/H}\sum_{h \in H}h\sigma(P_2) \\
&=&\sum_{h \in H}h(P_1)+\sum_{\sigma \in G_1}\sigma(P_2). 
\end{eqnarray*}
Similarly, 
$$ f_H^*\left(\overline{P_2}+\sum_{\overline{\tau} \in \overline{G_2}}\overline{\tau}(\overline{P_1})\right)=\sum_{h \in H}h(P_2)+\sum_{\tau \in G_2}\tau(P_1).$$
Condition (c') is satisfied. 
\end{proof}

\begin{proof}[Proof of Corollary \ref{maincor}]
By condition (d), $H \cap G_{i}=\{1\}$ for $i=1, 2$. 
By condition (e), $HG_i=H \rtimes G_i$. 
Let $\hat{G_i}=H \rtimes G_i$ for $i=1, 2$. 
Note that $H \triangleleft \hat{G_i}$ for $i=1, 2$. 
We would like to prove that conditions (a'), (b'), (c') and (d') are satisfied for the $5$-tuple $(\hat{G_1}, \hat{G_2}, H, P_1, P_2)$. 
Condition (f) is the same as condition (d'). 

Since $k(\mathcal{C})^{\hat{G_i}} \subset k(\mathcal{C})^{G_i}$, by condition (a) and L\"{u}roth's theorem, it follows that $\mathcal{C}/\hat{G_i} \cong \mathbb{P}^1$. 
Condition (a') is satisfied.  

Let $\eta \in \hat{G_1} \cap \hat{G_2}$. 
Then, there exist $h_1, h_2 \in H$, $\sigma \in G_1$ and $\tau \in G_2$ such that $\eta=h_1\sigma=h_2\tau$. 
Then, $\sigma\tau^{-1}=h_1^{-1}h_2 \in H$. 
By condition (d), $\sigma\tau^{-1}=1$ and hence, $\sigma=\tau \in G_1 \cap G_2$. 
By condition (b), $\sigma=\tau=1$. 
This implies that $\eta \in H$. 
It follows that $\hat{G_1} \cap \hat{G_2}=H$. 
Condition (b') is satisfied. 
 
By condition (c), it follows that
$$ P_1+\sum_{\sigma \in G_1}\sigma(P_2)=P_2+\sum_{\tau \in G_2}\tau(P_1). $$
For each $h \in H$, 
$$ h(P_1)+\sum_{\sigma \in G_1}h\sigma(P_2)=h(P_2)+\sum_{\tau \in G_2}h\tau(P_1). $$
Therefore, 
$$ \sum_{h \in H}h(P_1)+\sum_{h \in H}\sum_{\sigma \in G_1}h\sigma(P_2)=\sum_{h \in H}h(P_2)+\sum_{h \in H}\sum_{\tau \in G_2}h\tau(P_1). $$
Condition (c') is satisfied, since each element of $\hat{G_1}$ (resp. of $\hat{G_2}$) is represented as $h\sigma$ (resp. $h\tau$) for some $h \in H$ and $\sigma \in G_1$ (resp. $\tau \in G_2$). 
\end{proof}

\section{An application to cyclic subcovers of the Giulietti--Korchm\'{a}ros curve} 
Let $p>0$ and let $q$ be a power of $p$. 
We consider the Giulietti--Korchm\'{a}ros curve $\mathcal{X} \subset \mathbb{P}^3$, which is defined by 
$$ x^q+x-y^{q+1}=0 \ \mbox{ and } \ y((x^q+x)^{q-1}-1)-z^{q^2-q+1}=0 $$
(see \cite{giulietti-korchmaros}). 
The group
$$G_1:=\left\{
\left(\begin{array}{cccc}
1 & b^q &  0 &a \\
0 & 1 &  0 &b \\
0 & 0 & 1 & 0 \\
0 & 0 & 0 & 1
\end{array}\right)
\ | \ a, b \in \mathbb{F}_{q^2}, \ a^q+a-b^{q+1}=0\right\}
\subset {\rm PGL}(4, k)$$ 
of order $q^3$ acts on $\mathcal{X}$ and on the set $\mathcal{X} \cap \{Z=0\}=\mathcal{X}(\mathbb{F}_{q^2})$ of all $\mathbb{F}_{q^2}$-rational points of $\mathcal{X}$, and this group fixes a point $P_1:=(1:0:0:0) \in \mathcal{X}$ (see \cite{giulietti-korchmaros}). 
Let 
$$\xi(x, y, z)=\left(\frac{1}{x}, -\frac{y}{x}, \frac{z}{x}\right). $$
Then, $\xi$ acts on $\mathcal{X}$ and on $\mathcal{X}(\mathbb{F}_{q^2})$, and $P_2:=\xi(P_1)=(0:0:0:1)$. 
Let $G_2:=\xi G_1 \xi^{-1}$, which fixes $P_2$. 
According to \cite{fukasawa-higashine2}, the $4$-tuple $(G_1, G_2, P_1, P_2)$ satisfies conditions (a), (b) and (c). 

It is not difficult to check that the cyclic group
$$ C_{q^2-q+1}:=\left\{(x, y, z) \mapsto (x, y, \zeta z) \ | \ \zeta^{q^2-q+1}=1 \right\}$$
acts on $\mathcal{X}$. 
We prove the following. 

\begin{theorem} \label{gk}
Let $H$ be a subgroup of $C_{q^2-q+1}$. 
Then, there exists a birational embedding $\psi: \mathcal{X}/H \rightarrow \mathbb{P}^2$ of degree $q^3+1$ with two inner Galois points.
\end{theorem}

\begin{proof}
Note that $H$ fixes all points of $\mathcal{X}(\mathbb{F}_{q^2})$ ($=\mathcal{X} \cap \{Z=0\}$). 
Therefore, $HP_1=\{P_1\} \ne \{P_2\}=HP_2$. 
Since each element of $G_1G_2 \setminus \{1\}$ does not fix some point in $\mathcal{X}(\mathbb{F}_{q^2})$, $H \cap G_1G_2=\{1\}$ follows.  
It is easily verified that $HG_1=H \times G_1$. 
Since $\xi h=h \xi$ for each element $h \in H$, $HG_2=H \times G_2$ follows. 
Therefore, the $5$-tuple $(G_1, G_2, P_1, P_2, H)$ satisfies conditions (d), (e) and (f). 
By Corollary \ref{maincor}, the assertion holds. 
\end{proof}  

\section{The curves constructed by Skabelund and their quotient curves}

We consider the cyclic cover $\Tilde{\mathcal{S}}$ of the Suzuki curve $\mathcal{S}$, constructed by Skabelund (\cite{skabelund}). 
Let $p=2$, let $q_0$ be a power of $2$, and let $q=2q_0^2$. 
The curve $\Tilde{\mathcal{S}}$ is the smooth model of the curve defined by
$$ y^q+y=x^{q_0}(x^q+x) \ \mbox{ and } \ x^q+x=z^{q-2q_0+1} $$
in $\mathbb{P}^3$. 
Let $P_1 \in \Tilde{\mathcal{S}}$ be the pole of $x$. 
It is known that the group
$$G_1:=\left\{
\left(\begin{array}{cccc}
1 & 0 &  0 &a \\
a^{q_0} & 1 &  0 &b \\
0 & 0 & 1 & 0 \\
0 & 0 & 0 & 1
\end{array}\right)
\ | \ a, b \in \mathbb{F}_{q}\right\}
\subset {\rm PGL}(4, k)$$ 
of order $q^2$ acts on $\Tilde{\mathcal{S}}$ and on the set $\Tilde{\mathcal{S}}(\mathbb{F}_{q})$ of all $\mathbb{F}_{q}$-rational points of $\Tilde{\mathcal{S}}$, and this group fixes $P_1$ (see \cite{gmqz, skabelund}). 
Let $\alpha:=y^{2q_0}+x^{2q_0+1}$, $\beta:=xy^{2q_0}+\alpha^{2q_0}$ and let 
$$\xi(x, y, z)=\left(\frac{\alpha}{\beta}, \frac{y}{\beta}, \frac{z}{\beta}\right). $$
Then, $\xi$ acts on $\Tilde{\mathcal{S}}$ and on $\Tilde{\mathcal{S}}(\mathbb{F}_{q})$, and $P_2:=\xi(P_1)=(0:0:0:1)$. 
Let $G_2:=\xi G_1 \xi^{-1}$, which fixes $P_2$. 
Then, we have the following. 

\begin{theorem} \label{skabelund1}
The $4$-tuple $(G_1, G_2, P_1, P_2)$ satisfies conditions {\rm (a)}, {\rm (b)} and {\rm (c)}, that is, the curve $\Tilde{\mathcal{S}}$ has a plane model of degree $q^2+1$ with two inner Galois points. 
\end{theorem}

\begin{proof}
It is not difficult to check that $k(\Tilde{\mathcal{S}})^{G_1}=k(z)$. 
Since $G_1$ does not fix $P_2$, $G_1 \cap G_2=\{1\}$. 
Conditions (a) and (b) are satisfied.
Condition (c) is satisfied, since 
$$P_1+\sum_{\sigma \in G_1}\sigma(P_2)=\sum_{Q \in \Tilde{\mathcal{S}}(\mathbb{F}_q)}Q=P_2+\sum_{\tau \in G_2}\tau(P_1).  $$ 
\end{proof}

It is not difficult to check that the cyclic group
$$ C_{q-2q_0+1}:=\left\{(x, y, z) \mapsto (x, y, \zeta z) \ | \ \zeta^{q-2q_0+1}=1 \right\}$$
acts on $\mathcal{X}$. 
Similar to the proof of Theorem \ref{gk}, the following holds. 

\begin{theorem} \label{skabelund2}
Let $H$ be a subgroup of $C_{q-2q_0+1}$. 
Then, there exists a birational embedding $\psi: \Tilde{\mathcal{S}}/H \rightarrow \mathbb{P}^2$ of degree $q^2+1$ with two inner Galois points.
\end{theorem}

We consider the cyclic cover $\Tilde{\mathcal{R}}$ of the Ree curve $\mathcal{R}$, constructed by Skabelund. 
Let $p=3$, let $q_0$ be a power of $3$ and let $q=3q_0^2$. 
The curve $\Tilde{\mathcal{R}}$ is the smooth model of the curve defined by 
$$ y^q-y=x^{q_0}(x^q-x), \ z^q-z=x^{2q_0}(x^q-x) \ \mbox{ and } \ x^q-x=t^{q-3q_0+1}. $$
Let 
$$ C_{q-3q_0+1}:=\left\{(x, y, z) \mapsto (x, y, \zeta z) \ | \ \zeta^{q-3q_0+1}=1 \right\}. $$
Similar to Theorems \ref{skabelund1} and \ref{skabelund2}, the following result holds. 

\begin{theorem} \label{skabelund3} 
Let $H$ be a subgroup of $C_{q-3q_0+1}$. 
Then, the curves $\Tilde{\mathcal{R}}$ and $\Tilde{\mathcal{R}}/H$ have plane models of degree $q^3+1$ with two inner Galois points. 
\end{theorem}

\section{Relations with the previous works} 
We can provide another proof of Theorems 1 and 2 in \cite{fukasawa-higashine1}, by Corollary \ref{maincor} and the analysis of the Hermitian curve $\mathcal{H} \subset \mathbb{P}^2$: $X^qZ+XZ^q=Y^{q+1}$. 
We recover Theorem 1(1) in \cite{fukasawa-higashine1} here. 
Precisely: 

\begin{fact}
Let a positive integer $m$ divide $q+1$. 
Then, the smooth model of the curve $y^m=x^q+x$ possesses a birational embedding into $\mathbb{P}^2$ of degree $q+1$ with two inner Galois points. 
\end{fact} 

\begin{proof}
Let $P_1=(1:0:0)$ and $P_2=(0:0:1) \in \mathbb{P}^2$. 
Then, $P_1$ and $P_2$ are inner Galois points for the Hermitian curve $\mathcal{H} \subset \mathbb{P}^2$ (\cite{homma}). 
The associated Galois groups are represented by 
\begin{eqnarray*}
G_1&:=&\{(X: Y: Z) \mapsto (X+\alpha Z: Y: Z)\ | \ \alpha^q+\alpha=0\} \ \mbox{ and } \\ 
G_2&:=&\{(X: Y: Z) \mapsto (X: Y: Z+\alpha X)\ | \ \alpha^q+\alpha=0\} 
\end{eqnarray*}
respectively. 
Then, the $4$-tuple $(G_1, G_2, P_1, P_2)$ satisfies conditions (a), (b) and (c). 
Let $sm=q+1$ and let $C_s$ be a cyclic group of order $s$ generated by the automorphism group $(x, y) \mapsto (x, \zeta y)$, where $\zeta$ is a primitive $s$-th root of unity. 
Note that $C_s$ fixes all points in the line $Y=0$. 
Therefore, $C_sP_1=\{P_1\} \ne \{P_2\}=C_sP_2$. 
It is easily verified that $C_s \cap G_1G_2=\{1\}$ and $C_sG_i=C_s \times G_i$. 
Conditions (d), (e) and (f) are satisfied. 
By Corollary \ref{maincor}, the quotient curve $\mathcal{H}/C_s$ has a birational embedding of degree $q+1$ with two inner Galois points. 
On the other hand, the quotient curve $\mathcal{H}/C_s$ has a plane model defined by $y^m=x^q+x$. 
\end{proof}

A similar argument is applicable for the curve $\mathcal{C} \subset \mathbb{P}^2$ defined by $X^3Z+Y^4+Z^4=0$, which has two inner Galois points $P_1=(1:0:0)$ and $P_2=(-1:0:1)$ on the line $Y=0$ (under the assumption $p \ne 2, 3$), by taking $H=\langle \eta \rangle$ with $\eta(x, y)=(x, -y)$ (see also \cite{miura-yoshihara, yoshihara}). 
Here, the associated Galois groups $G_1$ and $G_2$ at $P_1$ and $P_2$ are generated by matrices
$$ \left(\begin{array}{ccc}
\omega & 0 & 0 \\
0 & 1 & 0 \\
0 & 0 & 1 
\end{array}\right) 
\ \mbox{ and } \
\left(\begin{array}{ccc} 
\frac{-\omega}{-\omega+1} & 0 & \frac{2}{-\omega+1} \\
0 & 1 & 0 \\
\frac{1}{-\omega+1} & 0 & \frac{\omega^2}{-\omega+1} 
\end{array}\right)
$$
respectively, where $\omega^2+\omega+1=0$. 
Then, the quotient curve $\mathcal{C}/H$ is the elliptic curve $y^2+x^3+1=0$. 
It is well known that $\mathcal{C}/H$ is isomorphic to the Fermat curve.  
Since the Galois group $G_i$ at $P_i$ fixes $P_i$, the group $\overline{G_i}$ fixes $\overline{P_i}$ for $i=1, 2$.   
Then, $\psi(\overline{P_i})$ is a total inflection point for this embedding $\psi$ (see \cite[III.8.2]{stichtenoth}). 
The following result is similar to \cite[Theorem 3]{fukasawa1}, but the proofs are different.

\begin{theorem}
Let $p \ne 2, 3$. 
For the cubic Fermat curve, there exists a plane model of degree four with two inner Galois points such that they are total inflection points. 
\end{theorem}

In \cite[Theorem 3]{fukasawa1}, we were not able to assert that two Galois points are inflection points, for the embedding provided in the proof. 

\begin{center} {\bf Acknowledgements} \end{center} 
The authors are grateful to Professor Takehiro Hasegawa for helpful conversations.

\end{document}